\documentclass[11pt]{amsart}

\usepackage[T1]{fontenc}
\usepackage[utf8]{inputenc}
\usepackage{amssymb}
\usepackage[hidelinks]{hyperref}
\usepackage[nameinlink,capitalize,noabbrev]{cleveref}

\hypersetup{
  pdftitle={The discrete second moment of the Riemann zeta-function along powers of Gram points},
  pdfauthor={Marius Dundulis; Ramūnas Garunkštis; Antanas Laurinčikas; Raivydas Šimėnas}
}

\numberwithin{equation}{section}
\newtheorem{theorem}{Theorem}[section]
\newtheorem{proposition}[theorem]{Proposition}
\newtheorem{lemma}[theorem]{Lemma}
\newtheorem{corollary}[theorem]{Corollary}
\theoremstyle{remark}
\newtheorem{remark}[theorem]{Remark}
\newtheorem*{introremark}{Remark 1.1}

\newcommand{\e}{\mathrm e}
\newcommand{\C}{\mathbb C}
\newcommand{\Sy}{\mathcal S}
\newcommand{\RePart}{\operatorname{Re}}
\newcommand{\ImPart}{\operatorname{Im}}

\begin{document}

\title[The second moment along powers of Gram points]
{The discrete second moment of the Riemann zeta-function\\
 along powers of Gram points}

\author[M. Dundulis]{Marius Dundulis}
\address[Marius Dundulis]{Euromonitor International, Jogailos g. 4,
LT-01116 Vilnius, Lithuania}
\email{marius.dundulis@euromonitor.com}

\author[R. Garunkštis]{Ramūnas Garunkštis}
\address[Ramūnas Garunkštis]{Institute of Mathematics, Faculty of Mathematics
and Informatics, Vilnius University, Naugarduko 24, 03225 Vilnius, Lithuania}
\email{ramunas.garunkstis@mif.vu.lt}
\urladdr{https://mif.vu.lt/~garunkstis/}

\author[A. Laurinčikas]{Antanas Laurinčikas}
\address[Antanas Laurinčikas]{Institute of Mathematics, Faculty of Mathematics
and Informatics, Vilnius University, Naugarduko 24, 03225 Vilnius, Lithuania}
\email{antanas.laurincikas@mif.vu.lt}

\author[R. Šimėnas]{Raivydas Šimėnas}
\address[Raivydas Šimėnas]{Institute of Mathematics, Faculty of Mathematics
and Informatics, Vilnius University, Naugarduko 24, 03225 Vilnius, Lithuania}
\email{raivydas.simenas@mif.vu.lt}
\urladdr{https://web.vu.lt/mif/r.simenas/}

\date{September 13, 2026}


\begin{abstract}
Let $t_k$ be the Gram points, and let $d\geq1$ be a fixed integer.  Put
\[
 \eta_d=\frac{1}{2^{d+1}-2},\qquad
 \sigma_d=1-\frac{\eta_d}{d},\qquad
 \sigma_d^*=\begin{cases}
  7/8,&d=2,\\
  \sigma_d,&d\ne2.
 \end{cases}
\]
We prove that, uniformly on compact subsets of
$\sigma_d^*<\RePart s<1$,
\[
 \frac1N\sum_{k=2}^{N}|\zeta(s+it_k^d)|^2
 \longrightarrow \zeta(2\RePart s).
\]
The underlying result is a mean-square approximation by finite Euler
products.  For the squares of Gram points, grouping equal frequency
ratios and applying the double large sieve improves the general
derivative-test boundary $11/12$ to $7/8$.  Assuming the Lindelöf
hypothesis, both
the Euler-product approximation and the second-moment formula hold on
every compact subset of $1/2<\RePart s<1$, for every fixed $d\geq1$.
As an application, we obtain a discrete universality theorem along the
same Gram-power shifts, in the same unconditional and conditional
regions.
\end{abstract}

\maketitle

\section{Introduction and main results}

The Riemann--Siegel theta-function is
\[
 \vartheta(t)=\ImPart\log\Gamma\left(\frac14+\frac{it}{2}\right)
              -\frac t2\log\pi,
\]
where the branch of the logarithm is chosen continuously.  For all
sufficiently large $t$, the function $\vartheta$ is strictly increasing.
We use the continuously interpolated Gram points $t_u$, defined by
\begin{equation}\label{eq:gram-definition}
 \vartheta(t_u)=\pi(u-1).
\end{equation}
At positive integral values $u=k$, these are the Gram points, with an
inessential index shift relative to another common convention.  For a
fixed integer $d\geq1$, we consider the shifts $t_k^d$.

Discrete questions involving Gram points arise naturally in the value
distribution of the zeta-function; see, for example,
\cite{Korolev2016,KorolevLaurincikas2019,Laurincikas2023}.
Related uniform-distribution results for the ordinates of $a$-points of
functional-equation factors, encompassing ordinary Gram points, were
obtained by Sourmelidis, Steuding, and Suriajaya
\cite[Corollary 1]{SourmelidisSteudingSuriajaya2022Indag}.
For $d\geq2$, the shifts $t_k^d$ are superlinear, and a continuous
mean-square theorem
cannot simply be sampled at the integers: differentiation of a
composition with the map $u\mapsto t_u^d$ introduces an unbounded
derivative.
We therefore estimate the discrete exponential sums directly.  This is
also the central theme in work on discrete moments along polynomial
shifts; see Sourmelidis \cite{Sourmelidis2020}.

The higher derivatives of $t_u$ and the qualitative uniform distribution
of $t_k^d$ modulo one were established by Dundulis, Garunk\v{s}tis,
Laurin\v{c}ikas, and \v{S}im\.{e}nas in the companion preprint
\cite{DundulisEtAl2026GramDerivatives}.  That preprint treats the
mean-square Euler-product approximation as a conjectural input.  In the
present paper we use its Gram-power derivative asymptotic to obtain a
quantitative exponential-sum estimate, from which we prove the
Euler-product approximation, the second moment, and universality in the
regions stated below.  

For $Y\geq2$, define the finite Euler product
\begin{equation}\label{eq:euler-product}
 \zeta_Y(s):=\prod_{p\leq Y}(1-p^{-s})^{-1}.
\end{equation}
Let $\Sy_Y$ be the set of positive integers all of whose prime factors
are at most $Y$.  Since only finitely many primes occur, one has the
absolutely convergent expansion
\begin{equation}\label{eq:smooth-series}
 \zeta_Y(s)=\sum_{m\in\Sy_Y}m^{-s},\qquad \RePart s>0.
\end{equation}

For a fixed integer $d\geq1$, set
\begin{equation}\label{eq:main-parameters}
 \eta_d:=\frac{1}{2^{d+1}-2},\qquad
 \sigma_d:=1-\frac{\eta_d}{d},\qquad
 \sigma_d^*:=\begin{cases}
  7/8,&d=2,\\
  \sigma_d,&d\ne2.
 \end{cases}
\end{equation}
The value $\sigma_d$ is supplied by the general derivative-test
argument.  Our baseline mean-square approximation is as follows.

\begin{theorem}[Unconditional Euler-product approximation]
\label{thm:unconditional-approximation}
Let $d\geq1$ be fixed and let $K$ be a compact set with
\[
 K\subset\{s\in\C:\sigma_d<\RePart s<1\}.
\]
Then
\begin{equation}\label{eq:main-approximation}
 \lim_{Y\to\infty}\limsup_{N\to\infty}\sup_{s\in K}
 \frac1N\sum_{k=2}^{N}
 \left|\zeta(s+it_k^d)-\zeta_Y(s+it_k^d)\right|^2=0.
\end{equation}
\end{theorem}

The quadratic case admits a stronger range because equal frequency
ratios can be grouped before the off-diagonal is estimated.

\begin{theorem}[Quadratic improvement]
\label{thm:quadratic-improvement}
Let $K$ be a compact set with
\[
 K\subset\{s\in\C:\tfrac78<\RePart s<1\}.
\]
Then \eqref{eq:main-approximation} holds with $d=2$.
\end{theorem}

The linear case has $\sigma_1=1/2$.  The next three left endpoints
supplied by the theorem are
\begin{equation}\label{eq:first-boundaries}
 \sigma_2=\frac{11}{12},\qquad
 \sigma_3=\frac{41}{42},\qquad
 \sigma_4=\frac{119}{120}.
\end{equation}
Here $11/12$ is the baseline value for $d=2$;
\cref{thm:quadratic-improvement} replaces it by $7/8$.  The method does
not include the boundary line $\RePart s=7/8$.

\begin{theorem}[Conditional Euler-product approximation]
\label{thm:conditional-approximation}
Assume the Lindelöf hypothesis.  Let $d\geq1$ be fixed and let $K$ be a
compact set with
\[
 K\subset\{s\in\C:\tfrac12<\RePart s<1\}.
\]
Then \eqref{eq:main-approximation} holds.
\end{theorem}

The approximation theorems imply the desired second-moment formula.

\begin{corollary}[Discrete second moment]
\label{cor:second-moment}
Let $d\geq1$ be fixed.  If
$K\subset\{s\in\C:\sigma_d^*<\RePart s<1\}$ is compact, then
\begin{equation}\label{eq:moment-limit}
 \lim_{N\to\infty}\sup_{s\in K}
 \left|\frac1N\sum_{k=2}^{N}|\zeta(s+it_k^d)|^2
       -\zeta(2\RePart s)\right|=0.
\end{equation}
Assuming the Lindelöf hypothesis, the same conclusion holds for every
compact set
$K\subset\{s\in\C:1/2<\RePart s<1\}$.
\end{corollary}

\begin{theorem}[Discrete universality along Gram powers]
\label{thm:universality}
Let $d\geq1$ be fixed, let
$K\subset\{s\in\C:\sigma_d^*<\RePart s<1\}$ be a compact set with
connected complement,
and let $f$ be continuous and nonvanishing on $K$ and analytic in its
interior.  Then, for every $\varepsilon>0$,
\begin{equation}\label{eq:universality}
 \liminf_{N\to\infty}\frac1N\#\left\{2\leq k\leq N:
 \max_{s\in K}|\zeta(s+it_k^d)-f(s)|<\varepsilon\right\}>0.
\end{equation}
Assuming the Lindelöf hypothesis, the same statement holds for every
compact set
$K\subset\{s\in\C:1/2<\RePart s<1\}$ with connected complement.
\end{theorem}

  The rest of the paper gives proofs of these statements.

\begin{introremark}
We note an interesting coincidence: the thresholds $7/8$ and $11/12$
appearing in our quadratic case also occur in OpenAI's preprints
\href{https://github.com/openai/math/blob/main/preprints/The-Quasi-Riemann-Hypothesis-September-30-2026/paper.pdf}
{\emph{The Quasi-Riemann Hypothesis: A Zero-Free Half-Plane
$\operatorname{Re}s>7/8$}} and
\href{https://github.com/openai/math/blob/main/preprints/The-Quasi-Riemann-Hypothesis-October-5-2026/paper2.pdf}
{\emph{The Quasi-Riemann Hypothesis}}, respectively.  These manuscripts
are dated 30 September and 5 October 2026 and were publicly released on
6 October 2026.

The proofs of our Theorems~1.1 and~1.2 were developed with assistance from
ChatGPT in exchanges on 21 August 2026: the $11/12$ threshold arose in
the exchange recorded at 6:46~p.m., and the $7/8$ threshold in those
recorded at 8:20~p.m.\ and 9:47~p.m.  

When subsequently asked about the
coincidence, ChatGPT gave the following assessment: ``There is a genuine
similarity in the use of frequency separation and large-sieve estimates,
but the matching $7/8$ and $11/12$ thresholds arise from different
calculations.''
\end{introremark}

\section{Exponential sums along Gram powers}

We use only two classical size estimates and one higher-derivative
asymptotic.  The results of Lavrik \cite{Lavrik1994} and Korolev
\cite{Korolev2016} give, as $u\to\infty$,
\begin{equation}\label{eq:t-asymptotic}
 \begin{aligned}
 t_u&=\frac{2\pi u}{\log u}
 \left(1+O\left(\frac{\log\log u}{\log u}\right)\right),\\
 t_u'&=\frac{2\pi}{\log u}
 \left(1+O\left(\frac{\log\log u}{\log u}\right)\right).
 \end{aligned}
\end{equation}
In particular, $t_u'/t_u\ll1/u$.
For every fixed integer $d\geq1$, Dundulis et al.
\cite[proof of Theorem~4]{DundulisEtAl2026GramDerivatives} proved
\begin{equation}\label{eq:phase-d-plus-one}
 \frac{d^{d+1}}{du^{d+1}}\bigl(t_u^d\bigr)
 \sim-\frac{d\,d!(2\pi)^d}{u\log^{d+1}u}.
\end{equation}
In particular, throughout every sufficiently large dyadic interval
$[X,2X]$,
\begin{equation}\label{eq:phase-derivative-comparison}
 \left|\frac{d^{d+1}}{du^{d+1}}\bigl(t_u^d\bigr)\right|
 \asymp_d\frac1{X(\log X)^{d+1}}.
\end{equation}

We use $e(x)=\exp(2\pi ix)$.  The following is a standard consequence
of the classical higher-derivative estimate; see Graham and Kolesnik
\cite[Theorems 2.2 and 2.8]{GrahamKolesnik1991}.

\begin{lemma}[Higher-derivative test]
\label{lem:higher-derivative-test}
Let $r\geq2$, let $I$ be an interval of length at most $H$, and suppose
that $F\in C^r(I)$ satisfies
\[
 \Lambda\leq |F^{(r)}(u)|\leq C\Lambda\qquad(u\in I)
\]
for a fixed $C\geq1$.  If $\delta_r=(2^r-2)^{-1}$, then
\begin{equation}\label{eq:higher-derivative-test}
 \left|\sum_{k\in I\cap\mathbb Z}e(F(k))\right|
 \ll_{r,C}H\Lambda^{\delta_r}
 +H^{1-2^{2-r}}\Lambda^{-\delta_r}+1.
\end{equation}
\end{lemma}

\begin{proof}
Since $F^{(r)}$ never vanishes, it has a constant sign on $I$.  The
case $r=2$ is the classical second-derivative estimate
\[
 \left|\sum_{k\in I\cap\mathbb Z}e(F(k))\right|
 \ll_C |I|\Lambda^{1/2}+\Lambda^{-1/2}+1.
\]
Replacing $|I|$ by $H$ proves \eqref{eq:higher-derivative-test} in
this case.  
For $r\geq3$, put $N=|I|\geq1$ and $Q=2^{r-2}$, so that
$\delta_r=(4Q-2)^{-1}$.  By
\cite[Theorem 2.8]{GrahamKolesnik1991},
\[
 \left|\sum_{k\in I\cap\mathbb Z}e(F(k))\right|
 \ll_{r,C}N\Lambda^{\delta_r}+N^{1-1/(2Q)}
 +N^{1-2/Q+1/Q^2}\Lambda^{-1/(2Q)}.
\]
The middle term is the geometric mean of $N\Lambda^{\delta_r}$ and
$N^{1-1/Q}\Lambda^{-\delta_r}$.  If $\Lambda\geq N^{-4+2/Q}$, the
last term is at most $N^{1-1/Q}\Lambda^{-\delta_r}$; otherwise this
quantity is at least $N$, and the trivial bound applies.  Hence
\[
 \left|\sum_{k\in I\cap\mathbb Z}e(F(k))\right|
 \ll_{r,C}|I|\Lambda^{1/(2^r-2)}
 +|I|^{1-2^{2-r}}\Lambda^{-1/(2^r-2)}+1.
\]
 Replacing $|I|$ by its upper bound $H$ again gives
\eqref{eq:higher-derivative-test}; the final $1$ also covers intervals
containing at most one integer.
\end{proof}

\begin{proposition}[Gram-power exponential sum]
\label{prop:gram-exponential-sum}
Let $d\geq1$ be fixed, let $L=\log X$, and let
$I\subset[X,2X]$ be any interval.  Uniformly for real $\lambda\ne0$,
\begin{align}
 \left|\sum_{k\in I\cap\mathbb Z}\e^{i\lambda t_k^d}\right|
 \ll_d{}&X^{1-\eta_d}|\lambda|^{\eta_d}
 L^{-(d+1)\eta_d}
 \label{eq:gram-exponential-sum}\\
 &+X^{1-2^{1-d}+\eta_d}|\lambda|^{-\eta_d}
 L^{(d+1)\eta_d}+1.\nonumber
\end{align}
\end{proposition}

\begin{proof}
By \eqref{eq:phase-d-plus-one}, throughout $[X,2X]$ the absolute value
of the $(d+1)$-st derivative of
$F(u)=\lambda t_u^d/(2\pi)$ is comparable with
\[
 \Lambda=\frac{|\lambda|}{XL^{d+1}}.
\]
Apply \cref{lem:higher-derivative-test} with $r=d+1$ and
$\delta_{d+1}=\eta_d$.  Since the derivative comparison holds on the
whole dyadic block, the same estimate is valid for every subinterval
$I\subset[X,2X]$.
\end{proof}

We record separately the elementary pair sums used after expanding a
Dirichlet-polynomial square.

\begin{lemma}[Logarithmic pair sums]
\label{lem:pair-sums}
Fix $0<\eta<1$, $\varepsilon>0$, and a closed interval
$J\subset(1/2,1)$.  Uniformly for $\sigma\in J$ and $M\geq2$,
\begin{align}
 \sum_{m<n\leq M}(mn)^{-\sigma}|\log(n/m)|^{\pm\eta}
 &\ll_{\eta,\varepsilon,J}M^{2-2\sigma+\varepsilon},
 \label{eq:pair-primal}\\
 \sum_{m<n\leq M}(mn)^{\sigma-1}|\log(n/m)|^{\pm\eta}
 &\ll_{\eta,\varepsilon,J}M^{2\sigma+\varepsilon}.
 \label{eq:pair-dual}
\end{align}
\end{lemma}

\begin{proof}
First consider $m<n\leq2m$ and write $n=m+h$.  Then
$\log(n/m)\asymp h/m$.  For the negative logarithmic power,
\[
 \sum_{m\leq M}m^{-2\sigma}
 \sum_{h\leq m}\left(\frac mh\right)^\eta
 \ll\sum_{m\leq M}m^{1-2\sigma}
 \ll_J M^{2-2\sigma}.
\]
The same calculation with $m^{2\sigma-2}$ in place of
$m^{-2\sigma}$ gives $O_J(M^{2\sigma})$ for the dual weights.
The positive logarithmic power is no larger in this near-diagonal
range.  In the complementary range $n>2m$, the negative logarithmic
power is bounded, while the positive power is $O((\log M)^\eta)$.
Separating the variables and using
$(\log M)^\eta\ll_\varepsilon M^\varepsilon$ proves both estimates.
\end{proof}

We shall also use the following one-dimensional form of the
Bombieri--Iwaniec double large sieve; see
Huxley \cite[Lemma~5.6.6]{Huxley1996}.

\begin{lemma}[Double large sieve]
\label{lem:double-large-sieve}
Let $|x_r|\leq R$ and $|y_v|\leq S$, and let $\alpha_r,\beta_v$ be
complex numbers.  Then
\begin{align}
 \left|\sum_r\sum_v\alpha_r\beta_v e(x_ry_v)\right|^2
 \ll{}&(1+RS)
 \sum_{|x_r-x_{r'}|\leq S^{-1}}|\alpha_r\alpha_{r'}|
 \nonumber\\
 &\times
 \sum_{|y_v-y_{v'}|\leq R^{-1}}|\beta_v\beta_{v'}|.
 \label{eq:double-large-sieve}
\end{align}
The same estimate holds after translating either set of points.
\end{lemma}

\begin{proof}
This is the cited inequality in dimension one.  Translating the
$x$-points or the $y$-points inserts factors of modulus one.  These can
be absorbed into $\beta_v$ or $\alpha_r$, respectively.
\end{proof}

\section{Approximate functional equations}

Put
\begin{equation}\label{eq:chi-definition}
 \chi(z)=\pi^{z-1/2}
 \frac{\Gamma((1-z)/2)}{\Gamma(z/2)},
 \qquad \zeta(z)=\chi(z)\zeta(1-z).
\end{equation}
We use the classical symmetric approximate functional equation.

\begin{lemma}[Symmetric approximate functional equation]
\label{lem:symmetric-afe}
Let $J\subset(0,1)$ be a closed interval, let
 $z=\sigma+iT$ with $\sigma\in J$ and $T\geq2$, and put
 $M=(T/2\pi)^{1/2}$.  Uniformly for $\sigma\in J$,
\begin{equation}\label{eq:symmetric-afe}
 \zeta(z)=\sum_{m\leq M}m^{-z}
 +\chi(z)\sum_{m\leq M}m^{z-1}
 +O_J(T^{-\sigma/2}).
\end{equation}
Here and below, a sharp cutoff $m\leq M$ means
$m\leq\lfloor M\rfloor$.
Moreover,
\begin{align}
 |\chi(\sigma+iT)|^2
 &=\left(\frac{T}{2\pi}\right)^{1-2\sigma}
 \left(1+O_J(T^{-1})\right),
 \label{eq:chi-size}\\
 \frac{\partial}{\partial T}|\chi(\sigma+iT)|^2
 &\ll_J T^{-1}|\chi(\sigma+iT)|^2.
 \label{eq:chi-variation}
\end{align}
\end{lemma}

\begin{proof}
The Hardy--Littlewood approximate functional equation, in its uniform
form, states that whenever $x,y\geq1$ and $2\pi xy=T$,
\[
 \zeta(z)=\sum_{m\leq x}m^{-z}
 +\chi(z)\sum_{m\leq y}m^{z-1}
 +O_J(x^{-\sigma})+O_J(T^{1/2-\sigma}y^{\sigma-1});
\]
 see Titchmarsh \cite[formula (4.12.4), Theorem 4.15]{Titchmarsh1986},
 or the formulation
used by Sourmelidis \cite[Section 2]{Sourmelidis2020}.  Taking
$x=y=(T/2\pi)^{1/2}$ makes both error terms $O_J(T^{-\sigma/2})$.
The two remaining estimates follow by applying Stirling's formula to
the quotient of gamma factors in \eqref{eq:chi-definition} and then
differentiating its logarithm.  The constants are uniform on the fixed
closed interval $J$.
\end{proof}

Under the Lindelöf hypothesis we shall use a much shorter Dirichlet
polynomial; see also Sourmelidis, Steuding, and Suriajaya
\cite[Section 1.1, formula (5)]{SourmelidisSteudingSuriajaya2022Acta}.

\begin{lemma}[Short polynomial under Lindelöf]
\label{lem:lindelof-polynomial}
Assume the Lindelöf hypothesis.  Fix $\sigma_0>1/2$ and
$0<\mu<1$.  Then there exists $\nu>0$ such that
\begin{equation}\label{eq:lindelof-polynomial}
 \zeta(\sigma+iT)=\sum_{m\leq T^\mu}m^{-\sigma-iT}
 +O_{\sigma_0,\mu}(T^{-\nu})
\end{equation}
uniformly for $\sigma_0\leq\sigma\leq1$ and all sufficiently large
$T$.
\end{lemma}

\begin{proof}
By Gonek \cite[Theorem 2.1 and formula (12.1)]{Gonek2012}, under the
Lindelöf hypothesis, for every $\epsilon>0$,
\[
 \zeta(\sigma+iT)=\sum_{m\leq X}m^{-\sigma-iT}
 +O_\epsilon\bigl(X^{1/2-\sigma}(T+2)^\epsilon\bigr),
\]
uniformly for $1/2\leq\sigma\leq1$, $2\leq X\leq(T+2)^2$, and all
sufficiently large $T$.  Take $X=T^\mu$ and
$\epsilon=\mu(\sigma_0-1/2)/2$.  Then the error is
$O_{\sigma_0,\mu}(T^{-\nu})$ uniformly for
$\sigma_0\leq\sigma\leq1$, where
$\nu=\mu(\sigma_0-1/2)/2>0$.
\end{proof}

\section{The unconditional approximation}

\begin{proof}[Proof of \cref{thm:unconditional-approximation}]
Choose
\begin{equation}\label{eq:sigma-zero-choice}
 \sigma_d<\sigma_0<\min_{s\in K}\RePart s,
\end{equation}
and write $s=\sigma+i\tau$.  The quantities $\sigma$ and $\tau$ range
over fixed bounded intervals.  On a dyadic block $X<k\leq2X$, put
\[
 z_k=s+it_k^d,\qquad T_k=\ImPart z_k=\tau+t_k^d,
 \qquad M_k=\left(\frac{T_k}{2\pi}\right)^{1/2}.
\]
By \eqref{eq:t-asymptotic}, uniformly for $s\in K$,
\begin{equation}\label{eq:block-sizes}
 T_k\asymp_K\frac{X^d}{(\log X)^d},\qquad
 M_k\asymp_K\frac{X^{d/2}}{(\log X)^{d/2}}.
\end{equation}
For large $X$, both $T_k$ and $M_k$ are increasing functions of $k$.
Let $M_X=C_KX^{d/2}(\log X)^{-d/2}$, with $C_K$ large enough that
$M_k\leq M_X$ on the block.  Since $C_K$ is fixed,
\eqref{eq:block-sizes} also gives the two-sided comparison
$M_k\asymp_K M_X$.

Apply \cref{lem:symmetric-afe} and subtract the absolutely convergent
series \eqref{eq:smooth-series}.  We obtain
\begin{equation}\label{eq:afe-decomposition}
 \zeta(z_k)-\zeta_Y(z_k)
 =A_{k,Y}(s)+R_{k,Y}(s)+C_k(s)+E_k(s),
\end{equation}
where
\begin{align}
 A_{k,Y}(s)&=\sum_{\substack{m\leq M_k\\m\notin\Sy_Y}}m^{-z_k},
 \label{eq:primal-definition}\\
 R_{k,Y}(s)&=-\sum_{\substack{m>M_k\\m\in\Sy_Y}}m^{-z_k},
 \label{eq:smooth-tail-definition}\\
 C_k(s)&=\chi(z_k)\sum_{m\leq M_k}m^{z_k-1},
 \label{eq:dual-definition}\\
 E_k(s)&\ll_K T_k^{-\sigma/2}.
 \label{eq:afe-error}
\end{align}
For fixed $Y$, absolute convergence in \eqref{eq:smooth-series} gives
\begin{equation}\label{eq:smooth-tail}
 \sup_{s\in K}\max_{X<k\leq2X}|R_{k,Y}(s)|
 \leq\sum_{\substack{m\in\Sy_Y\\m>c_KX^{d/2}(\log X)^{-d/2}}}
 m^{-\sigma_0}=o_{K,Y}(1).
\end{equation}
Also, \eqref{eq:block-sizes} and \eqref{eq:afe-error} show that the
block mean square of $E_k$ tends to zero.

We now estimate $A_{k,Y}$.  Put
\begin{equation}\label{eq:D-Y}
 D_Y:=\sum_{m\notin\Sy_Y}m^{-2\sigma_0}.
\end{equation}
Since $m\notin\Sy_Y$ implies $m>Y$,
\begin{equation}\label{eq:D-Y-decay}
 D_Y\leq\sum_{m>Y}m^{-2\sigma_0}
 \ll_{\sigma_0}Y^{1-2\sigma_0}\longrightarrow0.
\end{equation}
Upon expanding the square of \eqref{eq:primal-definition}, the diagonal
is bounded by $D_Y$.  For an off-diagonal pair $m<n\leq M_X$, the
condition $n\leq M_k$ selects either the empty set or a terminal integer
subinterval $I_n(s)$ of $(X,2X]$.  Hence
\cref{prop:gram-exponential-sum}, with
$\lambda=\log(n/m)$, applies to the $k$-sum.  The bounded number $\tau$
only contributes the unit factor $\exp(i\tau\log(n/m))$.

Using \cref{lem:pair-sums} with $\eta=\eta_d$ and absorbing powers of
$\log X$ into $X^\varepsilon$, we obtain, uniformly for $s\in K$,
\begin{align}
 \frac1X\sum_{X<k\leq2X}|A_{k,Y}(s)|^2
 \ll{}&D_Y
+X^{d(1-\sigma_0)-\eta_d+\varepsilon}
\label{eq:primal-mean-square}\\
&+X^{d(1-\sigma_0)-2^{1-d}+\eta_d+\varepsilon}
+X^{d(1-\sigma_0)-1+\varepsilon}.\nonumber
\end{align}
For clarity, the first off-diagonal term before logarithmic factors are
absorbed is
\[
 X^{-\eta_d}(\log X)^{-(d+1)\eta_d}
 M_X^{2-2\sigma_0+\varepsilon},
\]
which explains the exponent $d(1-\sigma_0)-\eta_d$ in
\eqref{eq:primal-mean-square}.

It remains to treat the dual sum.  Put
$w_s(u)=|\chi(s+it_u^d)|^2$.  From
\eqref{eq:chi-size}, \eqref{eq:chi-variation}, and
$\frac{d}{du}(t_u^d)/t_u^d\ll_d1/X$ on $[X,2X]$, one has
\begin{equation}\label{eq:dual-weight-variation}
 \sup_{X\leq u\leq2X}w_s(u)
 \ll_K M_X^{2-4\sigma},\qquad
 \int_X^{2X}|w_s'(u)|\,du
 \ll_K\sup_{X\leq u\leq2X}w_s(u).
\end{equation}
The first estimate here uses both sides of $M_k\asymp_KM_X$, because
$2-4\sigma<0$.  After the square of \eqref{eq:dual-definition} is
expanded, the cutoff $n\leq M_k$ gives a terminal interval $I_n(s)$.
Abel summation gives, for every real $\lambda$,
\begin{align}
 \left|\sum_{k\in I_n(s)}w_s(k)\e^{i\lambda t_k^d}\right|
 \leq{}&\left(\sup_{u\in I_n(s)}w_s(u)
 +\int_{I_n(s)}|w_s'(u)|\,du\right)\nonumber\\
 &\times\max_{J\subseteq I_n(s)}
 \left|\sum_{k\in J\cap\mathbb Z}\e^{i\lambda t_k^d}\right|.
 \label{eq:dual-abel}
\end{align}
Thus the all-subinterval estimate
\eqref{eq:gram-exponential-sum} applies to every dual off-diagonal pair.

The dual diagonal is, uniformly on $K$,
\begin{align*}
 \frac1X\sum_{X<k\leq2X}w_s(k)
 \sum_{m\leq M_k}m^{2\sigma-2}
 &\ll_K M_X^{2-4\sigma}M_X^{2\sigma-1}\\
 &\ll_K X^{d(1/2-\sigma_0)+\varepsilon}.
\end{align*}
For the off-diagonal, \eqref{eq:pair-dual} gives the factor
$M_X^{2\sigma+\varepsilon}$; multiplying by the first factor in
\eqref{eq:dual-weight-variation} leaves
$M_X^{2-2\sigma+\varepsilon}$.  Consequently,
\begin{align}
 \frac1X\sum_{X<k\leq2X}|C_k(s)|^2
 \ll{}&X^{d(1/2-\sigma_0)+\varepsilon}
 +X^{d(1-\sigma_0)-\eta_d+\varepsilon}
 \label{eq:dual-mean-square}\\
 &+X^{d(1-\sigma_0)-2^{1-d}+\eta_d+\varepsilon}
 +X^{d(1-\sigma_0)-1+\varepsilon}.\nonumber
\end{align}

The choice \eqref{eq:sigma-zero-choice} says
$d(1-\sigma_0)<\eta_d$.  Moreover,
\begin{equation}\label{eq:savings-comparison}
 2^{1-d}-\eta_d\geq\eta_d\qquad(d\geq1).
\end{equation}
Therefore every power of $X$ in
\eqref{eq:primal-mean-square} and
\eqref{eq:dual-mean-square} is negative when $\varepsilon>0$ is chosen
sufficiently small.  By \eqref{eq:afe-decomposition}, the elementary
inequality $|a+b+c+d|^2\leq4(|a|^2+|b|^2+|c|^2+|d|^2)$, and
\eqref{eq:smooth-tail},
\begin{equation}\label{eq:dyadic-approximation}
 \sup_{s\in K}\frac1X\sum_{X<k\leq2X}
 |\zeta(z_k)-\zeta_Y(z_k)|^2
 \ll_K D_Y+o_{K,Y}(1).
\end{equation}

Decompose $[2,N]$ into dyadic blocks.  The displayed negative powers
give some fixed $0<\kappa<1$.  Before normalization, a subinterval
$I\subset(X,2X]$ contributes at most
\[
 |I|D_Y+O_K(X^{1-\kappa})+o_{K,Y}(1)|I|;
\]
the diagonal scales with $|I|$, while the off-diagonal bound is valid
on every subinterval by \cref{prop:gram-exponential-sum} and
\eqref{eq:dual-abel}.  This also treats the final partial block.
Blocks below a fixed threshold contribute $o(1)$ after division by
$N$, and $\sum_{2^j\leq N}2^{j(1-\kappa)}\ll N^{1-\kappa}$.
Consequently,
\[
 \limsup_{N\to\infty}\sup_{s\in K}\frac1N\sum_{k=2}^{N}
 |\zeta(s+it_k^d)-\zeta_Y(s+it_k^d)|^2\ll_K D_Y.
\]
Finally let $Y\to\infty$ and use \eqref{eq:D-Y-decay}.
\end{proof}

\begin{remark}[The baseline value $11/12$]
For $d=2$, \eqref{eq:phase-d-plus-one} becomes
\[
 (t_u^2)'''=-\frac{16\pi^2}{u\log^3u}
 \left(1+o(1)\right),
\]
and \eqref{eq:gram-exponential-sum} reads
\[
 \sum_{k\in I}\e^{i\lambda t_k^2}
 \ll \frac{X^{5/6}|\lambda|^{1/6}}{\sqrt{\log X}}
 +\frac{X^{2/3}\sqrt{\log X}}{|\lambda|^{1/6}}+1.
\]
Since the symmetric Dirichlet polynomial has length
$M_X\asymp X/\log X$, its first normalized off-diagonal contribution
is $X^{11/6-2\sigma_0+\varepsilon}$.  It tends to zero precisely when
$\sigma_0>11/12$.  This is the quadratic instance of
$d(1-\sigma_0)<\eta_d$.  The next section gains an additional saving
by first grouping all pairs with the same reduced ratio.
\end{remark}

\section{The quadratic improvement}

\begin{proof}[Proof of \cref{thm:quadratic-improvement}]
Fix
\[
 \frac78<\sigma_0<\min_{s\in K}\RePart s,
 \qquad L=\log X,
 \qquad M=M_X\asymp_K\frac XL,
 \qquad D=\frac{X}{L^2}.
\]
On $X<k\leq2X$ retain the notation $z_k,T_k,M_k,D_Y$ from the
preceding section, now with $d=2$.  Choose
\begin{equation}\label{eq:quadratic-short-length}
 0<\theta<\frac1{12},\qquad
 u_0=\frac34-\theta,\qquad U=X^{u_0}.
\end{equation}
Partition $(X,2X]$ into consecutive intervals $J$ of length between
$U/2$ and $U$, apart from at most one shorter final interval.

The sharp cutoff $M_k$ changes slowly on $J$.  Indeed,
\eqref{eq:t-asymptotic} and the definition of $M_k$ give
\[
 \frac{d}{du}\left(\frac{\tau+t_u^2}{2\pi}\right)^{1/2}
 \ll_K\frac1L\qquad(X\leq u\leq2X).
\]
Put $M_J=\min_{k\in J}\lfloor M_k\rfloor$.  Then
\begin{equation}\label{eq:quadratic-cutoff-change}
 \max_{k\in J}|\lfloor M_k\rfloor-M_J|
 \ll_K1+\frac UL.
\end{equation}
Replacing $M_k$ by $M_J$ in the primal polynomial
\eqref{eq:primal-definition} changes it pointwise by at most
\[
 \left(1+\frac UL\right)M^{-\sigma_0}
 \ll_K X^{u_0-\sigma_0+\varepsilon}.
\]
For the dual polynomial, \eqref{eq:chi-size} gives
$|\chi(z_k)|\ll_KM^{1-2\sigma}$, so its cutoff shell is bounded by
\[
 |\chi(z_k)|\left(1+\frac UL\right)M^{\sigma-1}
 \ll_K\left(1+\frac UL\right)M^{-\sigma}
 \ll_KX^{u_0-\sigma_0+\varepsilon}.
\]
This is $o_K(1)$ because $u_0<3/4<\sigma_0$.  Its contribution to
either block mean square is
\begin{equation}\label{eq:quadratic-shell-square}
 O_K\left(X^{2(u_0-\sigma_0)+\varepsilon}\right).
\end{equation}
It therefore remains to estimate the two polynomials with the fixed
cutoff $M_J$ on each $J$.

We begin with the primal polynomial.  Write every off-diagonal pair
uniquely as
\[
 n=ar,\qquad m=br,\qquad a>b,\qquad(a,b)=1.
\]
After all terms with the same ratio $a/b$ are grouped, their coefficient,
apart from the fixed unit factor $\exp(i\tau\log(a/b))$, is
\begin{equation}\label{eq:quadratic-primal-coefficient}
 c^+_{a,b,J}(s)
 =(ab)^{-\sigma}
 \sum_{\substack{r\leq M_J/a\\ar,br\notin\Sy_Y}}r^{-2\sigma}.
\end{equation}
Consequently,
\begin{equation}\label{eq:quadratic-primal-coefficient-bound}
 |c^+_{a,b,J}(s)|\ll_K(ab)^{-\sigma_0}.
\end{equation}

Put $a\asymp A$, $b\asymp B$, $A\geq B$, and $P=AB$.  Distinct
reduced fractions in this box satisfy
\begin{equation}\label{eq:quadratic-ratio-spacing}
 \left|\log\frac ab-\log\frac{a'}{b'}\right|\gg\frac1P.
\end{equation}
Indeed, $|ab'-a'b|\geq1$, and all ratios in the box are comparable.
After the logarithms are centered, they lie in an interval of bounded
length.  Center also $t_k^2/(2\pi)$ at one point of $J$.  By
\eqref{eq:t-asymptotic}, these centered phase points have span
\begin{equation}\label{eq:quadratic-phase-span}
 S_J\asymp DU
\end{equation}
and mutual spacing $\asymp D$.  Hence their close-pair sum in
\cref{lem:double-large-sieve} is diagonal and is $O(U)$.  On the
frequency side, \eqref{eq:quadratic-ratio-spacing} and
\eqref{eq:quadratic-primal-coefficient-bound} give
\[
 \sum_{|x_{a,b}-x_{a',b'}|\ll S_J^{-1}}
 |c^+_{a,b,J}(s)c^+_{a',b',J}(s)|
 \ll_KP^{1-2\sigma_0}\left(1+\frac{P}{S_J}\right),
\]
where $x_{a,b}=\log(a/b)$.  The double large sieve, summed over the
$O(X/U)$ intervals $J$ and normalized by $X$, therefore bounds this
dyadic ratio box by
\begin{equation}\label{eq:quadratic-primal-dls-box}
 \ll_KD^{1/2}P^{1/2-\sigma_0}
       +U^{-1/2}P^{1-\sigma_0}.
\end{equation}

There is a complementary pointwise estimate.  For
$\lambda=\log(a/b)$, equations
\eqref{eq:phase-derivative-comparison} and
\eqref{eq:higher-derivative-test}, used with the actual interval length
$U$ and $r=3$, give
\begin{align}
 \frac1U\left|\sum_{k\in J}\e^{i\lambda t_k^2}\right|
 \ll{}&X^{-1/6}L^{-1/2}|\lambda|^{1/6}
 \nonumber\\
 &+X^{1/6}U^{-1/2}L^{1/2}|\lambda|^{-1/6}+U^{-1}.
 \label{eq:quadratic-short-cubic}
\end{align}
The proof of \cref{lem:pair-sums}, applied to the coprime variables
$a,b$, shows that the sum of either logarithmic power, weighted by
$(ab)^{-\sigma_0}$, over the box is
$O_{K,\varepsilon}(P^{1-\sigma_0+\varepsilon})$.  For example, in the
only delicate case $A\asymp B$, writing $a=b+h$ gives
\[
 A^{-2\sigma_0}\sum_{b\asymp A}\sum_{1\leq h\ll A}
 \left(\frac Ah\right)^{1/6}
 \ll A^{2-2\sigma_0}\asymp P^{1-\sigma_0}.
\]
Because $u_0>2/3$, all three terms in
\eqref{eq:quadratic-short-cubic} yield the normalized box estimate
\begin{equation}\label{eq:quadratic-primal-pointwise-box}
 \ll_{K,\varepsilon}X^{-1/6+\varepsilon}P^{1-\sigma_0}.
\end{equation}

Use \eqref{eq:quadratic-primal-pointwise-box} for
$P\leq X^{4/3}$ and the first term of
\eqref{eq:quadratic-primal-dls-box} for $P>X^{4/3}$.  The two estimates
meet at $P=X^{4/3}$; since $1/2-\sigma_0<0$, dyadic summation gives
\begin{align}
 \frac1X\sum_{X<k\leq2X}|A_{k,Y}(s)|^2
 \ll_{K,\varepsilon}{}&D_Y
 +X^{-1/6+\frac43(1-\sigma_0)+\varepsilon}
 \nonumber\\
 &+U^{-1/2}M^{2-2\sigma_0+\varepsilon}
 +X^{2(u_0-\sigma_0)+\varepsilon}.
 \label{eq:quadratic-primal-bound}
\end{align}

For completeness, consider the dual polynomial with cutoff $M_J$.
For $a\asymp A$, $b\asymp B$, $A\geq B$, its ratio coefficient before
the factor $w_s(k)=|\chi(z_k)|^2$ is
\begin{equation}\label{eq:quadratic-dual-coefficient}
 c^-_{a,b,J}(s)
 =(ab)^{\sigma-1}\sum_{r\leq M_J/a}r^{2\sigma-2}.
\end{equation}
It satisfies
\[
 |c^-_{a,b,J}(s)|
 \ll_KM^{2\sigma-1}A^{-\sigma}B^{\sigma-1},
\]
and hence
\begin{equation}\label{eq:quadratic-dual-coefficient-square}
 \sum_{\substack{a\asymp A,\ b\asymp B\\(a,b)=1}}
 |c^-_{a,b,J}(s)|^2
 \ll_KM^{4\sigma-2}\left(\frac AB\right)^{1-2\sigma}.
\end{equation}
On $J$, \eqref{eq:chi-size} also gives
\[
 \sum_{k\in J}|w_s(k)|^2\ll_KUM^{4-8\sigma}.
\]
Thus the same double-large-sieve calculation, now with $w_s(k)$ as
the coefficient of the phase point, gives after normalization
\[
 \ll_KD^{1/2}M^{1-2\sigma_0}
       +U^{-1/2}M^{2-2\sigma_0+\varepsilon}.
\]
Here we used
$(A/B)^{1/2-\sigma_0}\leq1$ and
$A^{1-\sigma_0}B^{\sigma_0}\leq M$.  Adding the dual diagonal and the
cutoff shell gives
\begin{align}
 \frac1X\sum_{X<k\leq2X}|C_k(s)|^2
 \ll_{K,\varepsilon}{}&M^{1-2\sigma_0+\varepsilon}
 +D^{1/2}M^{1-2\sigma_0+\varepsilon}
 \nonumber\\
 &+U^{-1/2}M^{2-2\sigma_0+\varepsilon}
 +X^{2(u_0-\sigma_0)+\varepsilon}.
 \label{eq:quadratic-dual-bound}
\end{align}

Every term on the right of
\eqref{eq:quadratic-primal-bound} and
\eqref{eq:quadratic-dual-bound}, apart from $D_Y$, tends to zero after
$\varepsilon>0$ is chosen sufficiently small.  The decisive exponent is
\[
 -\frac16+\frac43(1-\sigma_0)
 =-\frac43\left(\sigma_0-\frac78\right)<0.
\]
The estimates for the smooth tail and for the error term in the
approximate functional equation are unchanged.  Therefore
\[
 \sup_{s\in K}\frac1X\sum_{X<k\leq2X}
 |\zeta(s+it_k^2)-\zeta_Y(s+it_k^2)|^2
 \ll_KD_Y+o_{K,Y}(1).
\]

The argument is uniform on subintervals of $(X,2X]$: partition such an
interval into the same $U$-blocks.  At most two remaining pieces have
length less than $U$ and are bounded trivially; after normalization
their contribution is
\[
 \ll_{K,\varepsilon}\frac UX M^{2-2\sigma_0+\varepsilon}
 =o_K(1).
\]
The dyadic passage at the end of the general proof now applies, and
letting $Y\to\infty$ proves \eqref{eq:main-approximation} for $d=2$.
\end{proof}

\begin{remark}
Within this double-large-sieve and pointwise hybrid, the decisive
crossover is unchanged for every admissible choice $2/3<u<3/4$.
Thus this argument stops at $7/8$.  We do not claim that the threshold
is optimal; moving it to the left would require an additional estimate,
for example a stronger short-interval exponential-sum bound.
\end{remark}

\section{The Lindelöf case}

\begin{proof}[Proof of \cref{thm:conditional-approximation}]
Choose
\[
 \frac12<\sigma_0<\min_{s\in K}\RePart s
\]
and then choose $\mu>0$ so small that
\begin{equation}\label{eq:mu-choice}
 0<\mu<\min\left\{\frac1d,
 \frac{\eta_d}{2d(1-\sigma_0)}\right\}.
\end{equation}
On $X<k\leq2X$, retain the notation
$z_k=s+it_k^d$ and $T_k=\ImPart z_k$.  By
\cref{lem:lindelof-polynomial},
\begin{equation}\label{eq:block-lindelof-polynomial}
 \zeta(z_k)=\sum_{m\leq T_k^\mu}m^{-z_k}+O_K(T_k^{-\nu}),
\end{equation}
uniformly for $s\in K$.  The length is
\begin{equation}\label{eq:conditional-length}
 T_k^\mu\asymp_K X^{d\mu}(\log X)^{-d\mu}=:M_X.
\end{equation}

Subtracting \eqref{eq:smooth-series}, we write
\begin{equation}\label{eq:conditional-decomposition}
 \zeta(z_k)-\zeta_Y(z_k)
 =A^{(\mu)}_{k,Y}(s)+R^{(\mu)}_{k,Y}(s)+O_K(T_k^{-\nu}),
\end{equation}
where
\begin{align*}
 A^{(\mu)}_{k,Y}(s)
 &=\sum_{\substack{m\leq T_k^\mu\\m\notin\Sy_Y}}m^{-z_k},\\
 R^{(\mu)}_{k,Y}(s)
 &=-\sum_{\substack{m>T_k^\mu\\m\in\Sy_Y}}m^{-z_k}.
\end{align*}
For fixed $Y$, the second sum is $o_{K,Y}(1)$ uniformly on the block,
by the same absolute-convergence argument as in
\eqref{eq:smooth-tail}.  The error in
\eqref{eq:block-lindelof-polynomial} is also negligible in block mean
square.

Expand the square of $A^{(\mu)}_{k,Y}$.  Its diagonal is at most $D_Y$.
For $m<n$, the condition $n\leq T_k^\mu$ again selects a terminal
subinterval, so \cref{prop:gram-exponential-sum} applies.  By
\cref{lem:pair-sums},
\begin{align}
 \frac1X\sum_{X<k\leq2X}|A^{(\mu)}_{k,Y}(s)|^2
 \ll{}&D_Y
+X^{-\eta_d+2d\mu(1-\sigma_0)+\varepsilon}
\label{eq:conditional-mean-square}\\
&+X^{-2^{1-d}+\eta_d+2d\mu(1-\sigma_0)+\varepsilon}
+X^{-1+2d\mu(1-\sigma_0)+\varepsilon}.\nonumber
\end{align}
The first exponent is decisive.  Indeed,
 \eqref{eq:mu-choice} gives
 $2d\mu(1-\sigma_0)<\eta_d$, while
 \eqref{eq:savings-comparison} shows that the second exponent is no
 larger.  Moreover,
 $2d\mu(1-\sigma_0)<\eta_d<1$, so the third exponent is negative as
 well.  Choose
$\varepsilon>0$ below the resulting fixed margin.  Then all three
terms tend to zero.

Combining \eqref{eq:conditional-decomposition} and
\eqref{eq:conditional-mean-square} gives the dyadic estimate
\[
 \sup_{s\in K}\frac1X\sum_{X<k\leq2X}
 |\zeta(z_k)-\zeta_Y(z_k)|^2
 \ll_K D_Y+o_{K,Y}(1).
\]
The three displayed negative exponents give a fixed power saving
$X^{-\kappa'}$.  Hence the subinterval and geometric-summation argument
at the end of the preceding section applies verbatim, first giving an
upper limit $\ll_KD_Y$ and then, as $Y\to\infty$,
\eqref{eq:main-approximation}.

For $d=2$, the safe choice \eqref{eq:mu-choice} reads
\[
 0<\mu<\min\left\{\frac12,
 \frac{1}{24(1-\sigma_0)}\right\}.
\]
Thus the same calculation covers every closed substrip to the right of
$1/2$ in the motivating quadratic case.
\end{proof}

\section{Passage to the second moment}

\begin{proof}[Proof of \cref{cor:second-moment}]
Fix $Y$.  By the absolute convergence of \eqref{eq:smooth-series},
\begin{align}
 |\zeta_Y(s+it_k^d)|^2
 =\sum_{m,n\in\Sy_Y}(mn)^{-\sigma}
 \exp\left(i(\tau+t_k^d)\log\frac nm\right),
 \label{eq:finite-product-expansion}
\end{align}
where $s=\sigma+i\tau$.  If $m\ne n$, then
$\lambda=\log(n/m)$ is fixed and nonzero.  Applying
\cref{prop:gram-exponential-sum} on dyadic blocks gives
\begin{equation}\label{eq:fixed-frequency-limit}
 \frac1N\sum_{k=2}^{N}\e^{it_k^d\lambda}\longrightarrow0.
\end{equation}
To justify the limit uniformly for $s\in K$, first truncate the
absolutely convergent double series in
\eqref{eq:finite-product-expansion}.  There are then only finitely many
nonzero frequencies, and \eqref{eq:fixed-frequency-limit} is uniform.
The remaining double-series tail is bounded uniformly by
\[
 \left(\sum_{\substack{m\in\Sy_Y\\m>M}}m^{-\sigma_K}\right)
 \left(2\sum_{m\in\Sy_Y}m^{-\sigma_K}
 +\sum_{\substack{m\in\Sy_Y\\m>M}}m^{-\sigma_K}\right),
 \qquad \sigma_K:=\min_{s\in K}\RePart s>\frac12,
\]
which tends to zero as $M\to\infty$.  Therefore
\begin{equation}\label{eq:finite-euler-moment}
 \frac1N\sum_{k=2}^{N}|\zeta_Y(s+it_k^d)|^2
 \longrightarrow
 C_Y(\sigma):=\sum_{m\in\Sy_Y}m^{-2\sigma}
 =\prod_{p\leq Y}(1-p^{-2\sigma})^{-1}
\end{equation}
uniformly for $s\in K$.

As $Y\to\infty$, one has
\begin{equation}\label{eq:diagonal-limit}
 C_Y(\sigma)\longrightarrow\zeta(2\sigma)
\end{equation}
uniformly for $\sigma\geq\sigma_K$, since the omitted terms are bounded
by $\sum_{m>Y}m^{-2\sigma_K}$.

It remains to transfer the limit from $\zeta_Y$ to $\zeta$.  Put
\begin{align*}
 a_N(s)&=\left(\frac1N\sum_{k=2}^{N}
 |\zeta(s+it_k^d)|^2\right)^{1/2},\\
 b_{N,Y}(s)&=\left(\frac1N\sum_{k=2}^{N}
 |\zeta_Y(s+it_k^d)|^2\right)^{1/2}.
\end{align*}
Also set
\[
 E_Y:=\limsup_{N\to\infty}\sup_{s\in K}\frac1N\sum_{k=2}^{N}
 |\zeta(s+it_k^d)-\zeta_Y(s+it_k^d)|^2.
\]
Minkowski's inequality gives
\begin{equation}\label{eq:minkowski-transfer}
 |a_N(s)-b_{N,Y}(s)|
 \leq\left(\frac1N\sum_{k=2}^{N}
 |\zeta(s+it_k^d)-\zeta_Y(s+it_k^d)|^2\right)^{1/2}.
\end{equation}
For fixed $Y$, the triangle inequality gives
\begin{align*}
 |a_N(s)-\sqrt{\zeta(2\RePart s)}|
 \leq{}&|a_N(s)-b_{N,Y}(s)|
 +|b_{N,Y}(s)-\sqrt{C_Y(\RePart s)}|\\
 &+|\sqrt{C_Y(\RePart s)}-\sqrt{\zeta(2\RePart s)}|.
\end{align*}
Take the upper limit as $N\to\infty$, uniformly on $K$.  The middle
term vanishes by \eqref{eq:finite-euler-moment}, and the first is
bounded by $E_Y^{1/2}$.  Now let $Y\to\infty$ and use
\cref{thm:unconditional-approximation}, or
\cref{thm:quadratic-improvement} when $d=2$, together with
\eqref{eq:diagonal-limit}.  It follows that
$a_N(s)\to\sqrt{\zeta(2\RePart s)}$ uniformly on $K$.  These quantities
are uniformly bounded for large $N$, so squaring proves
\eqref{eq:moment-limit}.

Under the Lindelöf hypothesis, use
\cref{thm:conditional-approximation} in
\eqref{eq:minkowski-transfer}; the remainder of the argument is
unchanged.
\end{proof}

\section{Discrete universality}

We prove \cref{thm:universality} using twisted finite Euler
products, multiple integrals over finite cubes, and uniform
distribution modulo one.  

If $\mathcal P$ is a finite set of primes and
$\boldsymbol\omega=(\omega_p)_{p\in\mathcal P}\in[0,1)^{\mathcal P}$,
put
\begin{equation}\label{eq:voronin-twisted-product}
 Z_{\mathcal P}(s,\boldsymbol\omega)
 :=\prod_{p\in\mathcal P}
 \left(1-\exp(-2\pi i\omega_p)p^{-s}\right)^{-1}.
\end{equation}
We write
\[
 d\boldsymbol\omega_{\mathcal P}
 :=\prod_{p\in\mathcal P}d\omega_p
\]
and $\operatorname{vol}_{\mathcal P}$ for ordinary Euclidean volume on
$[0,1)^{\mathcal P}$.  Thus every phase integral below is an ordinary
finite-dimensional iterated integral over unit intervals; no
normalizing factor is needed because the unit cube has volume one.
A bounded open set $G$ is called admissible if, for every $\delta>0$,
the outer parallel set
\[
 G_\delta:=\{s\in\C:\operatorname{dist}(s,G)<\delta\}
\]
has connected complement.

We use the following standard approximation input.

\begin{lemma}[Voronin's twisted-product lemma]
\label{lem:voronin-twisted-product}
Let $G$ be admissible with
$\overline G\subset\{s\in\C:1/2<\RePart s<1\}$, and let $F$ be analytic
and nonvanishing on a neighborhood of $\overline G$.  Given
$z>0$ and $\rho>0$, there are a finite set of primes $\mathcal M$ and
phases $\boldsymbol\theta\in[0,1)^{\mathcal M}$ such that
\[
 \{p:p\leq z\}\subset\mathcal M,
 \qquad
 \max_{s\in\overline G}
 |Z_{\mathcal M}(s,\boldsymbol\theta)-F(s)|<\rho.
\]
\end{lemma}

\begin{proof}
This is \cite[Lemma~2]{Pankowski2018}, specialized to the Dirichlet
character modulo $1$ and to the empty excluded set of primes.  The
increasing finite prime sets in that lemma exhaust all primes; hence a
sufficiently late set contains every $p\leq z$ and gives uniform error
less than $\rho$ on $\overline G$.
\end{proof}

The next lemma is the one-function, nonresonant case of the direct
finite-product argument in \cite[Lemma~4]{Pankowski2018}.  We give the
calculation because the density constant must be independent of the
later cutoff $Y$.

\begin{lemma}[Finite products with a uniform positive density]
\label{lem:voronin-finite-products}
Let $d\geq1$ be fixed, let
$K\subset\{s\in\C:\sigma_d^*<\RePart s<1\}$ be a compact set with
connected complement,
and let $f$ satisfy the hypotheses of \cref{thm:universality}.  For every
$\varepsilon>0$ there exist $Y_0\geq2$ and $c>0$ such that, for every
fixed $Y\geq Y_0$,
\begin{equation}\label{eq:finite-products-positive-density}
 \liminf_{N\to\infty}\frac1N\#\left\{2\leq k\leq N:
 \max_{s\in K}|\zeta_Y(s+it_k^d)-f(s)|<3\varepsilon/4\right\}
 \geq c.
\end{equation}
The constant $c$ is independent of $Y$.  The same assertion holds with
$\sigma_d^*$ replaced by $1/2$.
\end{lemma}

\begin{proof}
Because $\C \setminus K$ is connected and $f$ has no zeros on $K$, the
function $f$ has a continuous logarithm on $K$ which is analytic in its
interior.  Mergelyan's theorem, applied to this logarithm, gives a
polynomial $Q$ such that, with $F=\exp Q$,
\begin{equation}\label{eq:voronin-mergelyan}
 \max_{s\in K}|F(s)-f(s)|<\frac{\varepsilon}{4}.
\end{equation}
Choose a bounded open rectangle $G$ with
\begin{equation}\label{eq:voronin-domain}
 K\subset G,
 \qquad
 \overline G\subset\{s\in\C:\sigma_d^*<\RePart s<1\}.
\end{equation}
Put
\[
 \sigma_G=\min_{s\in\overline G}\RePart s>\frac12,
 \qquad
 B=1+\max_{s\in\overline G}|F(s)|.
\]
Write $s=\sigma+i\tau$ and $dA(s)=d\sigma\,d\tau$ for ordinary planar
area measure.  Put
\[
 r=\frac12\operatorname{dist}(K,\C\setminus G)>0.
\]
For every function $h$ analytic on a neighborhood of $\overline G$, the
mean-value inequality for $|h|^2$ gives
\begin{equation}\label{eq:voronin-bergman}
 \max_{s\in K}|h(s)|^2
 \leq \frac1{\pi r^2}\iint_G|h(s)|^2\,dA(s)
 =:C_{K,G}\iint_G|h(s)|^2\,dA(s).
\end{equation}
Indeed, the closed disc of radius $r$ about every point of $K$ lies in
$G$, and the first inequality follows by integrating $|h|^2$ over that
disc.  This is also commonly called the Bergman--Cauchy estimate.

Since $\sigma_G>1/2$,
\begin{equation}\label{eq:voronin-tail-product}
 R(z):=\prod_{p>z}(1-p^{-2\sigma_G})^{-1}-1
 \longrightarrow0
 \qquad(z\to\infty).
\end{equation}
Choose $z$ so large that
\begin{equation}\label{eq:voronin-z-choice}
 C_{K,G}\operatorname{area}(G)B^2R(z)<\frac{\varepsilon^2}{128}.
\end{equation}
Apply \cref{lem:voronin-twisted-product} with this $z$ and with an error
smaller than $\min\{1/4,\varepsilon/16\}$.  It gives a finite prime set
$\mathcal M\supset\{p:p\leq z\}$ and phases
$\boldsymbol\theta$.  By an arbitrarily small perturbation of their
representatives, if necessary, we may suppose that
$\boldsymbol\theta\in(0,1)^{\mathcal M}$.  Continuity of
\eqref{eq:voronin-twisted-product} then gives an ordinary rectangular
box with closure contained in $(0,1)^{\mathcal M}$,
\[
 D=\prod_{p\in\mathcal M}(\alpha_p,\beta_p)
 \subset(0,1)^{\mathcal M},
 \qquad \boldsymbol\theta\in D,
\]
such that, for every $\boldsymbol\omega_{\mathcal M}\in D$,
\begin{equation}\label{eq:voronin-core-approximation}
 \max_{s\in\overline G}
 |Z_{\mathcal M}(s,\boldsymbol\omega_{\mathcal M})-F(s)|
 <\frac{\varepsilon}{8},
 \qquad
 \max_{s\in\overline G}
 |Z_{\mathcal M}(s,\boldsymbol\omega_{\mathcal M})|
 \leq B.
\end{equation}
Let
\[
 a=\operatorname{vol}_{\mathcal M}(D)
   =\prod_{p\in\mathcal M}(\beta_p-\alpha_p)>0,
 \qquad
 Y_0=\max\mathcal M.
\]

Fix $Y\geq Y_0$, put
\[
 \mathcal Q_Y=\{p:p\leq Y\},
 \qquad
 \mathcal R_Y=\mathcal Q_Y\setminus\mathcal M,
\]
Since $\mathcal M\subset\mathcal Q_Y$,
\begin{equation}\label{eq:voronin-core-tail-factorization}
 Z_{\mathcal Q_Y}(s,\boldsymbol\omega)
 -Z_{\mathcal M}(s,\boldsymbol\omega_{\mathcal M})
 =Z_{\mathcal M}(s,\boldsymbol\omega_{\mathcal M})
 \left(Z_{\mathcal R_Y}(s,\boldsymbol\omega_{\mathcal R_Y})-1\right).
\end{equation}

For $s=\sigma+i\tau$, elementary orthogonality on $[0,1]$ gives
\begin{align}
 &\int_0^1\!\cdots\!\int_0^1
 |Z_{\mathcal R_Y}(s,\boldsymbol\omega)-1|^2\,
 d\boldsymbol\omega_{\mathcal R_Y}
 \nonumber\\
 &\qquad
 =\sum_{\substack{(\nu_p)\in\mathbb N_0^{\mathcal R_Y}\\
                  (\nu_p)\ne\boldsymbol0}}
   \prod_{p\in\mathcal R_Y}p^{-2\nu_p\sigma}
 =\prod_{p\in\mathcal R_Y}(1-p^{-2\sigma})^{-1}-1
 \leq R(z).
 \label{eq:voronin-tail-orthogonality}
\end{align}
Indeed, expand the absolutely convergent Dirichlet series of the finite
Euler product and integrate term by term.  For every pair of
nonnegative integers $r,\ell$,
\[
 \int_0^1\exp(2\pi i(r-\ell)\omega)\,d\omega
 =
 \begin{cases}
 1,&r=\ell,\\
 0,&r\ne\ell.
 \end{cases}
\]
Thus all unequal prime-power monomials disappear.  Every prime in
$\mathcal R_Y$ exceeds $z$, which proves the final inequality.  The
usual empty-product convention covers $\mathcal R_Y=\varnothing$.

Combining \eqref{eq:voronin-bergman}--%
\eqref{eq:voronin-tail-orthogonality} and using ordinary iterated
integration gives
\begin{align}
 &\int_D\int_{[0,1]^{\mathcal R_Y}}
 \max_{s\in K}|Z_{\mathcal Q_Y}(s,\boldsymbol\omega)
 -Z_{\mathcal M}(s,\boldsymbol\omega_{\mathcal M})|^2
 \,d\boldsymbol\omega_{\mathcal R_Y}
 \,d\boldsymbol\omega_{\mathcal M}
 \nonumber\\
 &\qquad
 \leq aC_{K,G}\operatorname{area}(G)B^2R(z)
 <\frac{a\varepsilon^2}{128}.
 \label{eq:voronin-tail-mean-square}
\end{align}
Let $\Omega_Y$ be the subset of
$D\times[0,1)^{\mathcal R_Y}$ on which
\[
 \max_{s\in K}|Z_{\mathcal Q_Y}(s,\boldsymbol\omega)
 -Z_{\mathcal M}(s,\boldsymbol\omega_{\mathcal M})|
 <\frac{\varepsilon}{4}.
\]
The elementary integral form of Chebyshev's inequality and
\eqref{eq:voronin-tail-mean-square} give explicitly
\[
 \operatorname{vol}_{\mathcal Q_Y}
 \bigl((D\times[0,1)^{\mathcal R_Y})\setminus\Omega_Y\bigr)
 <\frac{16}{\varepsilon^2}\,
   \frac{a\varepsilon^2}{128}
 =\frac a8.
\]
Since the containing rectangle has volume $a$, it follows that
\begin{equation}\label{eq:voronin-omega-volume}
 \operatorname{vol}_{\mathcal Q_Y}(\Omega_Y)
 >a-\frac a8=\frac{7a}{8}>\frac a2.
\end{equation}
On $\Omega_Y$, equations \eqref{eq:voronin-mergelyan} and
\eqref{eq:voronin-core-approximation} give
\begin{equation}\label{eq:voronin-finite-product-close}
 \max_{s\in K}|Z_{\mathcal Q_Y}(s,\boldsymbol\omega)-f(s)|
 <\frac{5\varepsilon}{8}<\frac{3\varepsilon}{4}.
\end{equation}

It remains to turn the ordinary volume in
\eqref{eq:voronin-omega-volume} into density of Gram-power shifts.  For
every finite set of primes $\mathcal Q$, the vectors
\begin{equation}\label{eq:voronin-prime-vectors}
 \left(\frac{t_k^d\log p}{2\pi}\right)_{p\in\mathcal Q}
\end{equation}
are uniformly distributed modulo one in $[0,1)^{\mathcal Q}$.  By
Weyl's criterion, a nontrivial integer vector leads to
\[
 \frac1N\sum_{k=2}^{N}
 \exp\left(it_k^d\sum_{p\in\mathcal Q}h_p\log p\right),
 \qquad h_p\in\mathbb Z.
\]
Unique factorization shows that the frequency is nonzero, and
\cref{prop:gram-exponential-sum}, followed by a dyadic decomposition,
makes this average tend to zero.

Since $\Omega_Y$ is relatively open and has volume greater than $a/2$,
inner approximation supplies a finite union $\mathcal B_Y$ of boxes
with pairwise disjoint interiors and boundaries of volume zero such
that
\[
 \overline{\mathcal B_Y}\subset\Omega_Y,
 \qquad
 \operatorname{vol}_{\mathcal Q_Y}(\mathcal B_Y)>\frac a2.
\]
Uniform distribution now gives
\[
 \lim_{N\to\infty}\frac1N
 \#\left\{2\leq k\leq N:
 \left(\left\{\frac{t_k^d\log p}{2\pi}\right\}\right)_{p\leq Y}
 \in\mathcal B_Y\right\}
 =\operatorname{vol}_{\mathcal Q_Y}(\mathcal B_Y)>\frac a2.
\]
Finally,
\begin{equation}\label{eq:voronin-shift-identity}
 Z_{\mathcal Q_Y}\left(
 s,\left(\left\{\frac{t_k^d\log p}{2\pi}\right\}\right)_{p\leq Y}\right)
 =\zeta_Y(s+it_k^d).
\end{equation}
Equations \eqref{eq:voronin-finite-product-close} and
\eqref{eq:voronin-shift-identity} prove
\eqref{eq:finite-products-positive-density} with $c=a/2$.  The number
$a$, and hence $c$, was fixed before $Y$ and is independent of it.

For the last assertion, choose the rectangle $G$ with
$\overline G\subset\{1/2<\RePart s<1\}$.  Nothing else in the proof
changes.
\end{proof}

\begin{proof}[Proof of \cref{thm:universality}]
We first upgrade the mean-square approximation to the uniform norm on
$K$.  Choose a bounded open rectangle $G$ satisfying
\eqref{eq:voronin-domain}.  For every $k$, the function
\[
 h_{k,Y}(z)=\zeta(z+it_k^d)-\zeta_Y(z+it_k^d)
\]
is analytic on a neighborhood of $\overline G$, because $\overline G$
lies strictly to the left of the line $\RePart z=1$.  Apply
\eqref{eq:voronin-bergman}, average over $k$, and use Fubini's theorem.
With
\begin{equation}\label{eq:voronin-function-space-approximation}
 \Delta(Y):=\limsup_{N\to\infty}\frac1N\sum_{k=2}^{N}
 \max_{s\in K}|\zeta(s+it_k^d)-\zeta_Y(s+it_k^d)|^2,
\end{equation}
we obtain
\begin{align*}
 \Delta(Y)
 &\leq C_{K,G}\operatorname{area}(G)
 \limsup_{N\to\infty}\sup_{z\in\overline G}\frac1N\sum_{k=2}^{N}
 |\zeta(z+it_k^d)-\zeta_Y(z+it_k^d)|^2\\
 &\longrightarrow0
 \qquad(Y\to\infty)
\end{align*}
by \cref{thm:unconditional-approximation}, or by
\cref{thm:quadratic-improvement} when $d=2$, applied to $\overline G$.

Apply \cref{lem:voronin-finite-products} with the given
$\varepsilon$.  It supplies $Y_0$ and $c>0$, independent of the later
cutoff.  Only now choose one fixed $Y\geq Y_0$ so large that
\begin{equation}\label{eq:voronin-final-Y-choice}
 \Delta(Y)<\frac{c\varepsilon^2}{144}.
\end{equation}
By Chebyshev's inequality,
\begin{align}
 \limsup_{N\to\infty}\frac1N\#\biggl\{2\leq k\leq N:
 &\max_{s\in K}|\zeta(s+it_k^d)-\zeta_Y(s+it_k^d)|
 \geq\frac{\varepsilon}{4}\biggr\}
 <\frac c9.
 \label{eq:voronin-final-chebyshev}
\end{align}
On the other hand, \eqref{eq:finite-products-positive-density} gives a
set of lower density at least $c$ on which
$\max_{s\in K}|\zeta_Y(s+it_k^d)-f(s)|<3\varepsilon/4$.
After the exceptional set in \eqref{eq:voronin-final-chebyshev} is
removed, the remaining set has lower density at least $8c/9>0$, and the
triangle inequality gives
\[
 \max_{s\in K}|\zeta(s+it_k^d)-f(s)|<\varepsilon.
\]
This proves \eqref{eq:universality} by a direct density argument.

Under the Lindelöf hypothesis, take
$\overline G\subset\{1/2<\RePart s<1\}$ and use
\cref{thm:conditional-approximation} in the Bergman estimate following
\eqref{eq:voronin-function-space-approximation}, together with the last
assertion of \cref{lem:voronin-finite-products}.  The remaining
finite-product calculation and equidistribution argument are unchanged.
\end{proof}

\bibliographystyle{amsplain}
\bibliography{DiscreteGramMoment}

@book{GrahamKolesnik1991,
  author    = {Graham, S. W. and Kolesnik, G.},
  title     = {Van der {C}orput's Method of Exponential Sums},
  series    = {London Mathematical Society Lecture Note Series},
  volume    = {126},
  publisher = {Cambridge University Press},
  address   = {Cambridge},
  year      = {1991},
  doi       = {10.1017/CBO9780511661976}
}

@book{Huxley1996,
  author    = {Huxley, M. N.},
  title     = {Area, Lattice Points, and Exponential Sums},
  series    = {London Mathematical Society Monographs. New Series},
  volume    = {13},
  publisher = {Clarendon Press, Oxford University Press},
  address   = {Oxford},
  year      = {1996},
  doi       = {10.1093/oso/9780198534662.001.0001}
}

@misc{DundulisEtAl2026GramDerivatives,
  author       = {Dundulis, M. and Garunk{\v{s}}tis, R. and
                  Laurin{\v{c}}ikas, A. and {\v{S}}im{\.{e}}nas, R.},
  title        = {Higher derivatives of the {G}ram function},
  year         = {2026},
  howpublished = {\href{https://arxiv.org/abs/2610.17349}{arXiv:2610.17349} [math.NT]},
  doi          = {10.48550/arXiv.2610.17349}
}

@article{Korolev2016,
  author  = {Korolev, M. A.},
  title   = {Gram's law in the theory of the {R}iemann zeta-function. {P}art 1},
  journal = {Proc. Steklov Inst. Math.},
  volume  = {292},
  number  = {Suppl. 2},
  year    = {2016},
  pages   = {S1--S146},
  doi     = {10.1134/S0081543816030019}
}

@article{KorolevLaurincikas2019,
  author  = {Korolev, M. A. and Laurin{\v{c}}ikas, A.},
  title   = {A new application of the {G}ram points},
  journal = {Aequationes Math.},
  volume  = {93},
  number  = {5},
  year    = {2019},
  pages   = {859--873},
  doi     = {10.1007/s00010-019-00647-8}
}

@article{Laurincikas2023,
  author  = {Laurin{\v{c}}ikas, A.},
  title   = {Joint discrete approximation of analytic functions by shifts of the
             {R}iemann zeta-function twisted by {G}ram points},
  journal = {Mathematics},
  volume  = {11},
  number  = {3},
  year    = {2023},
  pages   = {565},
  doi     = {10.3390/math11030565}
}

@article{Lavrik1994,
  author  = {Lavrik, A. A.},
  title   = {Titchmarsh's problem in the discrete theory of the {R}iemann
             zeta-function},
  journal = {Proc. Steklov Inst. Math.},
  volume  = {207},
  year    = {1995},
  pages   = {179--209},
  note    = {Russian original published in 1994}
}

@misc{Sourmelidis2020,
  author       = {Sourmelidis, A.},
  title        = {Discrete moments of the {R}iemann zeta-function near
                  {$\sigma=1$}},
  year         = {2020},
  howpublished = {arXiv:2002.02772 [math.NT]},
  doi          = {10.48550/arXiv.2002.02772}
}

@book{Titchmarsh1986,
  author    = {Titchmarsh, E. C.},
  title     = {The Theory of the {R}iemann Zeta-Function},
  edition   = {2nd},
  note      = {Revised by D. R. Heath-Brown},
  publisher = {Oxford University Press},
  address   = {Oxford},
  year      = {1986}
}

@article{Pankowski2018,
  author  = {Pa{\'n}kowski, {\L}.},
  title   = {Joint universality for dependent {$L$}-functions},
  journal = {Ramanujan J.},
  volume  = {45},
  number  = {1},
  pages   = {181--195},
  year    = {2018},
  doi     = {10.1007/s11139-017-9886-5},
  eprint  = {1604.04396},
  archivePrefix = {arXiv}
}

@article{Gonek2012,
  author  = {Gonek, S. M.},
  title   = {Finite {E}uler products and the {R}iemann hypothesis},
  journal = {Trans. Amer. Math. Soc.},
  volume  = {364},
  number  = {4},
  year    = {2012},
  pages   = {2157--2191},
  doi     = {10.1090/S0002-9947-2011-05546-7}
}

@article{SourmelidisSteudingSuriajaya2022Acta,
  author  = {Sourmelidis, A. and Steuding, J. and Suriajaya, A. I.},
  title   = {{R}iemann-type functional equations: {D}irichlet polynomial
             approximations and a weak {G}ram law},
  journal = {Acta Arith.},
  volume  = {204},
  number  = {2},
  year    = {2022},
  pages   = {97--113},
  doi     = {10.4064/aa210111-13-4}
}

@article{SourmelidisSteudingSuriajaya2022Indag,
  author  = {Sourmelidis, A. and Steuding, J. and Suriajaya, A. I.},
  title   = {{R}iemann-type functional equations: {J}ulia line and
             counting formulae},
  journal = {Indag. Math. (N.S.)},
  volume  = {33},
  number  = {6},
  year    = {2022},
  pages   = {1236--1262},
  doi     = {10.1016/j.indag.2022.08.002}
}

\end{document}